\documentclass[reqno, oneside]{amsart}
\usepackage{amsthm}
\usepackage[a4paper]{geometry}
\usepackage{microtype}
\usepackage[english]{babel}
\usepackage{amsfonts}
\usepackage{bm}
\usepackage{latexsym}
\usepackage{amsmath}
\usepackage{amssymb}
\usepackage{array}
\usepackage{arydshln}
\usepackage{cases}
\usepackage{tabularx}
\usepackage{tikz}
\tikzset{
  font={\fontsize{9pt}{12}\selectfont}}
\usepackage{graphicx}
\numberwithin{equation}{section}
\newtheorem{theorem}{Theorem}

\newtheorem{lemma}{Lemma}
\newtheorem{corollary}{Corollary}

\newtheorem{proposition}{Proposition}
\newtheorem{conjecture}{Conjecture}

\theoremstyle{definition}
\newtheorem{definition}{Definition}
\newtheorem{remark}{Remark}

\title{On the $k$-Abelian complexity of the Cantor sequence}
\thanks{This work was supported by NSFC (Nos. 11401188 and 11626110).}
\author{Jin Chen}
\address[J. Chen]{College of Science, Huazhong Agricultural University, Wuhan 430070, China.}
\email{wind.golden@gmail.com}
\author{Xiaotao L\"{u}}
\address[X.T. L\"{u}]{School of Mathematics and Statistics, Huazhong University of Science and Technology, Wuhan 430074, China.}
\email{M201270021@hust.edu.cn}
\author{Wen Wu$^*$}\thanks{$^*$ Wen Wu is the corresponding author.}
\address[W. Wu]{School of Mathematics, South China University of Technology, Guangzhou 510641, China.}
\email{wuwen@scut.edu.cn}
\keywords{Cantor sequence; $k$-abelian complexity; $k$-regular sequences}
\begin{document}
\begin{abstract}
In this paper, we prove that for every integer $k\geq 1$, the $k$-abelian complexity  function of the Cantor sequence $\mathbf{c}=101000101\cdots$ is a $3$-regular sequence.
\end{abstract}
\subjclass[2010]{68R15 and 11B85} 
\maketitle
%\tableofcontents{section}
\section{Introduction}

This paper is devoted to the study of the $k$-abelian complexity of the Cantor sequence \[\mathbf{c}:=c_{0}c_{1}c_{2}\cdots=101000101000000000101000101\cdots\]
which satisfies $c_{0}=1$ and for all $n\geq 0$,
\begin{equation}\label{eq:rec}
c_{3n}=c_{3n+2}=c_{n} \text{ and } c_{3n+1}=0.
\end{equation}

The $k$-abelian complexity, which was introduced by Karhum\"{a}ki in \cite{K81}, is a measure of disorder of infinite words. It has been studied widely in \cite{F15,PRR15,KSZ14,CKS15,KSZ13}. 
Before we give its definition, we need some notations. Let $\mathcal{A}$ be a finite alphabet and $\mathcal{A}^{n}$ be the set of words of length $n$ for every positive integer $n$.  Denote $\mathcal{A}^{*}$ the set of all finite words on $\mathcal{A}$. For two words $u,v\in\mathcal{A}^{*}$, $v$ is called a \emph{factor} of $u$ if $u=wvw^{\prime}$ where $w,w^{\prime}\in\mathcal{A}^{*}$. For a word $u=u_{0}u_{1}\cdots u_{n-1}\in\mathcal{A}^{n}$, the \emph{prefix} and \emph{suffix} of length $\ell（\geq 1）$ are defined as \[\mathrm{pref}_{\ell}(u):=u_{0}u_{1}\cdots u_{\ell-1} \text{ and }
\mathrm{suff}_{\ell}(u):=u_{n-\ell}\cdots u_{n-1};\]
while for $\ell\leq 0$, we define $\mathrm{pref}_{\ell}(u)=\varepsilon$ and $\mathrm{suff}_{\ell}(u)=\varepsilon$, where $\varepsilon$ is the empty word.
Denote $|u|$ the length of a word $u$ and denote $|u|_{v}$ the number of occurrences of a word $v$ in $u$. 
\begin{definition}[see \cite{PRV14}]
Let $k\geq 1$ be an integer. Two words $u,v\in\mathcal{A}^{*}$ are called \emph{$k$-abelian equivalent}, written by $u\sim_{k}v$, if $\mathrm{pref}_{k-1}(u)=\mathrm{pref}_{k-1}(v)$, $\mathrm{suff}_{k-1}(u)=\mathrm{suff}_{k-1}(v)$ and $|u|_{w}=|v|_{w}$ for every $w\in\mathcal{A}^{k}$.
\end{definition}

The above definition is one of the equivalent definitions of the $k$-abelian equivalence; see also \cite{KSZ13}. The $k$-abelian equivalence is in fact an equivalence relation. The \emph{$k$-abelian complexity} of an infinite word $\omega$ is the function $\mathcal{P}^{(k)}_{\omega}:\mathbb{N}\rightarrow \mathbb{N}$ and for every $n\geq 1$, $\mathcal{P}^{(k)}_{\omega}(n)$ is assigned to be the number of $k$-abelian equivalence classes of factors of $\omega$ of length $n$. Precisely, for every positive integer $n$, 
\[\mathcal{P}^{(k)}_{\omega}(n)=\mathrm{Card}\big(\mathcal{F}_{\omega}(n)/\sim_{k}\big),\]
where $\mathcal{F}_{\omega}(n)$ is the set of all factors of length $n$ occurring in $\omega$. 

\medskip

In our first result, we reduce the $k$-abelian equivalence of any two factors of $\mathbf{c}$ to the abelian equivalence of such factors. In detail, we prove the following theorem.

\begin{theorem}\label{thm:B}
Let $k\geq 1$ be an integer and let $u,v$ be two factors of $\mathbf{c}$ satisfying $|u|=|v|$. If $\mathrm{pref}_{k}(u)=\mathrm{pref}_{k}(v)$ and $\mathrm{suff}_{k}(u)=\mathrm{suff}_{k}(v)$, then $u\sim_{k+1}v$ if and only if $u\sim_{1}v$.
\end{theorem}

By using Theorem \ref{thm:B}, we are able to study the $k$-abelian complexity of $\mathbf{c}$ for every $k\geq 1$, and we have the following result.
\begin{theorem}\label{thm:A}
For every integer $k\geq 1$, the $k$-abelian complexity function of the Cantor sequence is a $3$-regular sequence.
\end{theorem}

The $k$-regular sequence was introduced by Allouch and Shallit \cite{AS} as an extension of the $k$-automatic sequence. The definitions of the $k$-automatic sequences and the $k$-regular sequences are stated below; see also \cite{ALL, E71}. 
\begin{definition}
For an integer $k\geq 1$, a sequence $\mathbf{w}=(w_{n})_{n\geq 0}$ is a \emph{$k$-automatic} sequence if its \emph{$k$-kernel}
\[\mathcal{K}_{k}(\mathbf{w})=\{(w_{k^{e}n+c})_{n\geq 0}\mid e\geq 0, 0\leq c<k^{e}\}\]
is a finite set. The sequence $\bf{w}$ is called a \emph{$k$-regular} sequence if the $\mathbb{Z}$-module generated by its $k$-kernel is finitely generated.
\end{definition}

Karhum\"{a}ki, Saarela and Zamboni \cite{KSZ14} studied the $k$-abelian complexity of the Thue-Morse sequence, which is a $2$-automatic sequence. Vandomme, Parreau and Rigo \cite{PRV14} conjectured that the $2$-abelian complexity of the Thue-Morse sequence is a $2$-regular sequence. This has been proved  independently in \cite{F15} by Greinecker  and in \cite{PRR15} by  Parreau, Rigo, Rowland and Vandomme. 

Our result (Theorem \ref{thm:A}) supports the following more general conjecture, which has been posed in \cite{PRR15}.
\begin{conjecture}
The $k$-abelian complexity of any $\ell$-automatic sequence is an $\ell$-regular sequence. 
\end{conjecture}

This paper is organized as follows. In Section 2, we give the recurrence relations for the abelian complexity function of the sequence $\mathbf{c}$.  As a consequence, the
abelian complexity function of the Cantor sequence is a $3$-regular sequence. In Section 3, we prove Theorem \ref{thm:B}. In the last section, we give the proof of Theorem \ref{thm:A}.

\section{Abelian complexity}\label{sec:2}
The abelian complexity of an infinite word $\omega$ is in fact the $1$-abelian complexity of $\omega$. For more details of the abelian complexity,  see \cite{BR13,CR11,K81,RSZ11,SFR14,T10} and references therein. In this section, we shall investigate the abelian complexity of $\mathbf{c}.$ 

First we introduce a useful result which characterizes the left and right special factors of $\mathbf{c}.$ Recall that a factor $v$ of $w$ is called \emph{right special} (resp. \emph{left special}) if both $va$ and $vb$ (resp. $av$ and $bv$) are factors of $w$ for distinct letters $a,b\in \mathcal{A}$.  We denote $\mathcal{RS}_{w}(n)$ (resp. $\mathcal{LS}_{w}(n)$) the set of all right special (resp. left special) factors of $w$ of length $n$.
\begin{lemma}\label{lem:rs}
For every $i\geq 0$ and $3^{i}<k\leq 3^{i+1}$, \[\mathcal{RS}_{\mathbf{c}}(k)=\{0^k, \mathrm{suff}_{k}(\sigma^{i}(010))\}\text{ and }
\mathcal{LS}_{\mathbf{c}}(k)=\{0^k, \mathrm{pref}_{k}(\sigma^{i}(010))\}.\]
\end{lemma}
\begin{proof}
The result follows from \cite[Theorem 1]{LCGW16} and the fact that every left special factor in $\mathbf{c}$ is the reversal of some right special factor in $\mathbf{c}$.
\end{proof}

Let $\omega =\omega_{0}\omega_{1}\omega_{2}\cdots$ be an infinite sequence on $\{0, 1\}$. It is proved in \cite[Proposition 2.2]{BBT11} that the abelian complexity of $\omega$ is related to its digit sums in the following way: for every $n\geq 1$,
\begin{equation}
\mathcal{P}_{\omega}^{(1)}(n)=M_{\omega}(n)-m_{\omega}(n)+1, \label{eq:pm}
\end{equation}
where {\[M_{\omega}(n):=\max\left\{\Sigma_{j=i}^{i+n-1}\omega_{j}\mid i\geq 0\right\}\text{ and } m_{\omega}(n):=\min\left\{\Sigma_{j=i}^{i+n-1}\omega_{j} \mid i\geq 0\right\}.\]}
For the digit sums of the Cantor sequence $\mathbf{c}$, we have the following lemma.
\begin{lemma}\label{lem:ab}
For every integer $n\geq 1$, $M_{\mathbf{c}}(n)=\sum_{i=0}^{n-1}c_{i}$ and $m_{\mathbf{c}}(n)=0$.
\end{lemma}
\begin{proof}
Since $0^n$ is always a factor of $\mathbf{c}$ for every $n\geq 1$, we have $m_{\mathbf{c}}(n)=0$ for every $n\geq 1$.

For every $i\geq 0$ and $n\geq 1$, let $\Sigma(i,n):=\sum_{j=i}^{i+n-1}c_j$. We only need to show that $M_{\mathbf{c}}(n)\leq\Sigma(0,n)$ for every $n\geq 1$, since the inverse inequality always holds by definition.  For this purpose, we shall prove that for every $n\geq 1$,
\begin{equation}\label{eq:sigma}
\Sigma(i,n)\leq \Sigma(0,n) \text{ for every integer } i\geq 0.
\end{equation}
Since `$1$' occurs in $\mathbf{c}$ and `$11$' does not occur in $\mathbf{c}$, we have $\Sigma(i,1)\leq 1=\Sigma(0,1)$ and $\Sigma(i,2)\leq 1=\Sigma(0,2)$. Now suppose \eqref{eq:sigma} holds for $n<m$. We first deal with the case: $m=3j+2$.
By \eqref{eq:rec}, we have the following nine recurrence relations:
\[\left\{
\begin{array}{ll}
\Sigma(3i,3n)  = 2\Sigma(i,n), & \Sigma(3i+1,3n+2)  = \Sigma(i,n+1)+\Sigma(i+1,n), \\
\Sigma(3i,3n+1)  = \Sigma(i,n)+\Sigma(i,n+1), & \Sigma(3i+2,3n)  = \Sigma(i,n)+\Sigma(i+1,n),\\
\Sigma(3i,3n+2)  = \Sigma(i,n)+\Sigma(i,n+1), & \Sigma(3i+2,3n+1)  = \Sigma(i,n+1)+\Sigma(i+1,n),\\
\Sigma(3i+1,3n)  = \Sigma(i,n)+\Sigma(i+1,n), & \Sigma(3i+2,3n+2)  = \Sigma(i,n+1)+\Sigma(i+1,n+1),\\
\Sigma(3i+1,3n+1)  = \Sigma(i,n)+\Sigma(i+1,n). &
\end{array}\right.
\]
By the above equations and the inductive assumption, for every $i\geq 0$,
\begin{align*}
\Sigma(3i,3j+2)  & = \Sigma(i,j)+\Sigma(i,j+1)\leq \Sigma(0,j)+\Sigma(0,j+1)=\Sigma(0,3j+2),\\
\Sigma(3i+1,3j+2) &  = \Sigma(i+1,j)+\Sigma(i,j+1)\leq \Sigma(0,j)+\Sigma(0,j+1)=\Sigma(0,3j+2).
\end{align*}
Note that at least one of $c_{i}$ and $c_{i+1}$ must be zero. So
\[\Sigma(3i+2,3j+2)  = \Sigma(i,j+1)+\Sigma(i+1,j+1)\leq \Sigma(0,j)+\Sigma(0,j+1)=\Sigma(0,3j+2).\]
Therefore, \eqref{eq:sigma} holds in the case $m=3j+2$. Following the same way, we can verify \eqref{eq:sigma} when $m=3j,\, 3j+1$.
\end{proof}

\begin{corollary}\label{max:min}
$M_{\mathbf{c}}(1)=1$, $M_{\mathbf{c}}(2)=1$ and for every $n\geq 1$,
\[M_{\mathbf{c}}(3n)=2M_{\mathbf{c}}(n) \text{ and } M_{\mathbf{c}}(3n+1)=M_{\mathbf{c}}(3n+2)=M_{\mathbf{c}}(n)+M_{\mathbf{c}}(n+1).\]
Moreover, $\{M_{\mathbf{c}}(n)\}_{n\geq 1}$ is a $3$-regular sequence.
\end{corollary}

\begin{proposition}\label{abel:1}
$\mathcal{P}^{(1)}_{\mathbf{c}}(1)=2$, $\mathcal{P}^{(1)}_{\mathbf{c}}(2)=2$ and for every $n\geq 1$,
\[\mathcal{P}^{(1)}_{\mathbf{c}}(3n)=2\mathcal{P}^{(1)}_{\mathbf{c}}(n)-1 \text{ and } \mathcal{P}^{(1)}_{\mathbf{c}}(3n+1)=\mathcal{P}^{(1)}_{\mathbf{c}}(3n+2)=\mathcal{P}^{(1)}_{\mathbf{c}}(n)+\mathcal{P}^{(1)}_{\mathbf{c}}(n+1)-1.\]
Moreover, $\{\mathcal{P}^{(1)}_{\mathbf{c}}(n)\}_{n\geq 1}$ is a $3$-regular sequence.
\end{proposition}
\begin{proof}
It follows from Lemma \ref{lem:ab}, Corollary \ref{max:min} and  \eqref{eq:pm}.
\end{proof}

\section{From $k$-abelian equivalence to $1$-abelian equivalence}
In this section, we give a key theorem, which implies that under certain condition, $k$-abelian equivalence can be reduced to $1$-abelian equivalence.
Using this theorem, we deduce the regularity of the $k$-abelian complexity of $\mathbf{c}$ from that of the abelian complexity of $\mathbf{c}$. Before stating the result, we give two auxiliary lemmas. For $z,w \in \mathcal{A}^{*}$, we define
\[P(z,w):=
\begin{cases}
1, & \text{if }z\text{ is a prefix of }w,\\
0, & \text{otherwise,}
\end{cases}
\text{ and }
S(z,w):=
\begin{cases}
1, & \text{if } z \text{ is a suffix of }w,\\
0, & \text{otherwise.}
\end{cases}
\]
\begin{lemma}\label{pref:suff}
Let $\mathcal{\omega}\in \{0,1\}^{\mathbb{N}}$ and $u,z\in \mathcal{F}_{\mathcal{\omega}}$ with $|u|\geq |z|$. Suppose $z=ayb,$ where $a, b \in \{0,1\}$. We have
\[|u|_{z}=
\begin{cases}
 |u|_{ay}-S(ay ,u), & \text{if }~ay\notin RS_{\mathcal{\omega}},\\
 |u|_{yb}-P(yb,u), & \text{if }~yb\notin LS_{\mathcal{\omega}},\\
 |u|_{ay}-|u|_{ay(1-b)}-S(ay,u), & \text{if }~ay\in RS_{\mathcal{\omega}},\\
 |u|_{yb}-|u|_{(1-a)yb}-P(yb, u), & \text{if }~yb\in LS_{\mathcal{\omega}}.\\
\end{cases}
\]
\end{lemma}
\begin{proof}
Note that $|u|_{ay}-S(ay ,u)$ is the number of occurrences of a right extendable $ay$ in $u$. When $ay$ is not right special, every right extension of a right extendable $ay$ must be $z$. So, $|u|_{ay}-S(ay ,u)=|u|_{z}$.
When $ay$ is right special, its right extensions are either $z$ or $ay(1-b)$. So, $|u|_{ay}-S(ay ,u)=|u|_{z}+|u|_{ay(1-b)}$.
The rest cases can be verified in the same way.
\end{proof}

\begin{lemma}\label{Delta:i}
For every $i\geq 0$, $u\in \mathcal{F}_{\mathbf{c}}$, let $\Delta_{i}:=|u|_{0^{3^i+2}}+|u|_{10^{3^i}1}-|u|_{0^{3^i+1}}+1$. Then $\Delta_{i}\in\{0,1,2\}$ and
\[\Delta_{i}=
\begin{cases}
|u|_{0^{3^{i}1}}+\frac{2}{3^i}|u|_{0^{3^i+1}} +1-P(0^{3^i}1,u) ~(\mathrm{mod}~ 3), & \text{if }~P(0^{3^i+1},u)=S(0^{3^i+1},u)=0, \\
|u|_{0^{3^i}1}+1-S(0^{3^i+1},u)-P(0^{3^i}1,u) ~(\mathrm{mod}~ 2), & \text{otherwise}.
\end{cases}\]
\end{lemma}
\begin{proof}
Let $Z(\ell)$ $(\ell\geq 1)$ be the number of blocks of zeros (in $u$) of length not less than $\ell$. For example, when $u=0010100$, then $Z(1)=3$ and $Z(2)=2$. Note that, for every $\ell\geq 3^{i}+1$, $|0^{\ell}|_{0^{3^i+1}}-|0^{\ell}|_{0^{3^i+2}}=1$. So,
\begin{align*}
|u|_{0^{3^i+1}}-|u|_{0^{3^i+2}} &= \sum_{v\text{ is a block of zeros in }u}(|v|_{0^{3^i+1}}-|v|_{0^{3^i+2}})\\
& = \sum_{\substack{v\text{ is a block of zeros in }u\\ |v|\geq 3^{i}+1}}1=Z(3^{i}+1).
\end{align*}
 On the other hand, $10^{3^{i}}1$ only occurs in $\sigma^{i+1}(1)$. Thus, there is a block of zeros of length $3^{i+\ell}$ (for some $\ell\geq 1$) between two consecutive $10^{3^{i}}1$. Since the block of zeros could also be the prefix or suffix of $u$, we have $|u|_{10^{3^{i}}1}-1\leq Z(3^{i}+1)\leq |u|_{10^{3^{i}}1}+1$, which implies $\Delta_{i}\in\{0,1,2\}$.

When $P(0^{3^i+1},u)=1$ or $S(0^{3^i+1},u)=1$, there is at least one block of zeros of length not less than $3^{i}+1,$ which is not located between two consecutive $10^{3^{i}}1$. This implies that $|u|_{10^{3^{i}}1}\leq Z(3^{i}+1)\leq |u|_{10^{3^{i}}1}+1$. So, in this case, $\Delta_{i}\in\{0,1\}$.
Applying Lemma \ref{pref:suff} to $|u|_{0^{3^i+2}}$ and $|u|_{10^{3^i}1}$, we have
\begin{equation}\label{eq:delta1}
\Delta_{i}=|u|_{0^{3^i}1}-2|u|_{0^{3^i+1}1}+1-S(0^{3^i+1},u)-P(0^{3^i}1,u).
\end{equation}
Since $\Delta_{i}\in\{0,1\}$, by \eqref{eq:delta1}, $\Delta_{i}=|u|_{0^{3^i}1}+1-S(0^{3^i+1},u)-P(0^{3^i}1,u)~(\mathrm{mod}~2).$

Now, suppose $P(0^{3^i+1},u)=S(0^{3^i+1},u)=0$. Applying Lemma \ref{pref:suff} to $|u|_{0^{3^i+1}1}$, by \eqref{eq:delta1}, we have
\begin{equation}\label{eq:delta2}
\Delta_{i}=|u|_{0^{3^i}1}-2Z(3^i+1)+1-P(0^{3^i}1,u).
\end{equation}
Let $\sum_{v}$ denote the sum over all blocks of zeros $v$ of $u$ of length not less than $3^{i}+1$. Then
\begin{align*}
  |u|_{0^{3^{i} + 1}} & = \sum\nolimits_{v}|v|_{0^{3^i + 1}}
  = \sum\nolimits_{v}{(|v| - {3^i})}
  = \left(\sum\nolimits_{v}|v|\right) - {3^i}Z(3^{i}+1)
\end{align*}
Note that, in this case, all blocks of zeros of $u$ are of length $3^{i+\ell}$ for some $\ell\geq 1$. So,
\begin{equation}\label{eq:delta3}
-2Z(3^{i}+1)\equiv \frac{2}{3^{i}}|u|_{0^{3^{i} + 1}}~(\mathrm{mod}~3).
\end{equation}
The result of this case follows from \eqref{eq:delta3} and \eqref{eq:delta2}.
\end{proof}

Now, we prove Theorem \ref{thm:B}.
\begin{proof}[Proof of Theorem \ref{thm:B}]
{ Let $u,v\in\mathcal{F}_{\mathbf{c}}$ satisfying $|u|=|v|$, $\mathrm{pref}_{k}(u)=\mathrm{pref}_{k}(v)$ and $\mathrm{suff}_{k}(u)=\mathrm{suff}_{k}(v)$.} When $k\geq |u|$, the assumption gives $u=v$. In this case, the result is trivial. In the following, we alway assume that $k<|u|$.

The `only if' part follows directly from the definition of $k$-abelian equivalence. For the `if' part, we only need to show that $u\sim_{k}v$ implies that for every $z\in \mathcal{F}_{\mathbf{c}}(k+1)$, $|u|_{z}=|v|_{z}$. For this purpose, we separate $\mathcal{F}_{\mathbf{c}}(k+1)$ into two disjoint parts, i.e., $\mathcal{F}_{\mathbf{c}}(k+1)=E_{1}\cup E_{2}$, where
\begin{align*}
E_{1} & = \{z\in\mathcal{F}_{\mathbf{c}}(k+1)~|~ \mathrm{pref}_{k}(z)\notin\mathcal{RS}_{\mathbf{c}}(k) \text{ or } \mathrm{suff}_{k}(z)\notin\mathcal{LS}_{\mathbf{c}}(k)\},\\
E_{2} & = \{z\in\mathcal{F}_{\mathbf{c}}(k+1)~|~ \mathrm{pref}_{k}(z)\in\mathcal{RS}_{\mathbf{c}}(k) \text{ and } \mathrm{suff}_{k}(z)\in\mathcal{LS}_{\mathbf{c}}(k)\}.
\end{align*}
Suppose $z\in E_{1}$. If $\mathrm{pref}_{k}(z)\notin\mathcal{RS}_{\mathbf{c}}(k)$, then by Lemma \ref{pref:suff},
\[|u|_{z}=|u|_{\mathrm{pref}_{k}(z)}-S(\mathrm{pref}_{k}(z),u)=|v|_{\mathrm{pref}_{k}(z)}-S(\mathrm{pref}_{k}(z),v)=|v|_{z}.\]
If $\mathrm{suff}_{k}(z)\notin\mathcal{LS}_{\mathbf{c}}(k)$, then by Lemma \ref{pref:suff},
\[|u|_{z}=|u|_{\mathrm{suff}_{k}(z)}-P(\mathrm{suff}_{k}(z),u)=|v|_{\mathrm{suff}_{k}(z)}-P(\mathrm{suff}_{k}(z),v)=|v|_{z}.\]
So, for every $z\in E_{1}$, $|u|_{z}=|v|_{z}$.

Now, let $z\in E_{2}$. Suppose $3^{i}<k\leq 3^{i+1}$ for some $i\geq 0$. When $k\neq 3^{i}+1$, by Lemma \ref{lem:rs}, $E_{2}=\{0^{k+1}\}$.
By Lemma \ref{pref:suff} and the assumptions of this result,
\begin{align*}
|u|_{0^{k+1}} &= |u|_{0^{k}}-|u|_{0^{k}1}-S(0^{k},u)\\
&= |u|_{0^{k}}-\big(|u|_{0^{k-1}1}-P(0^{k-1}1,u)\big)-S(0^{k},u)\\
&= |v|_{0^{k}}-\big(|v|_{0^{k-1}1}-P(0^{k-1}1,v)\big)-S(0^{k},v) = |v|_{0^{k+1}}.
\end{align*}
When $k=3^{i}+1$, by Lemma \ref{lem:rs}, $E_{2}=\{0^{k+1}, 0^{k}1, 10^{k}, 10^{k-1}1\}$.
{ For every $w\in\mathcal{F}_{\mathbf{c}}$,
by Lemma \ref{pref:suff} and \ref{Delta:i}, we have the following linear system:
\begin{equation}\label{eq:3eq}
\left\{
\begin{aligned}
& |w|_{0^{k+1}}+|w|_{0^{k}1}  = |w|_{0^{k}}-S(0^{k},w),\\
& |w|_{0^{k+1}}+|w|_{10^{k}}  = |w|_{0^{k}}-P(0^{k},w),\\
& |w|_{10^{k}} +|w|_{10^{k-1}1}  = |w|_{10^{k-1}}-S(10^{k-1},w),\\
& |w|_{0^{k+1}}+|w|_{10^{k-1}1}  = |w|_{0^{k}}-1+\Delta_{i},
%|u|_{0^{k}1}+|u|_{10^{k-1}1} & = |u|_{0^{k-1}1}-P(0^{k-1}1,u).
\end{aligned}
\right.
\end{equation}
which determines $(|w|_{z})_{z\in E_{2}}$ uniquely.  If $u\sim_{k}v$, then the linear systems \eqref{eq:3eq} for $u$ and $v$ turn out to be the same one. So, $u\sim_{k}v$ implies $|u|_{z}=|v|_{z}$ for every factor  $z\in E_{2}$. }
\end{proof}

We may now apply Theorem \ref{thm:B} repeatedly to reduce the $k$-abelian equivalence to the $1$-abelian equivalence under the condition of Theorem \ref{thm:B}.
\begin{corollary}\label{abel:k}
Let $k\geq 1$ and $u,v\in \mathcal{F}_{\mathbf{c}}$ satisfying $|u|=|v|$. If $\mathrm{pref}_{k}(u)=\mathrm{pref}_{k}(v)$ and $\mathrm{suff}_{k}(u)=\mathrm{suff}_{k}(v)$, then $u\sim_{k+1} v ~\text{if and only if}~u\sim_{1} v$.
\end{corollary}

\begin{remark}
A similar result for Sturmian words is obtained by Karhum\"{a}ki, Saarela and Zamboni \cite[Corollary 3.1]{KSZ13}. We would like to ask that in general, what kind of infinite words share a property similar to Corollay \ref{abel:k}?
\end{remark}

\section{$k$-abelian complexity}
In this section, we first give the regularity of the $2$-abelian complexity of $\mathbf{c}$. Then, by using Theorem \ref{thm:B} properly, we deduce the regularity of the $k$-abelian complexity of $\mathbf{c}$. We start by classifying the $k$-abelian equivalent classes of $\mathcal{F}_\mathbf{c}(n)$ by their prefixes and suffixes of length  $k-1$. 

For every $k\geq 2$, $x,y\in\mathcal{F}_{\mathbf{c}}(k-1)$ and every $n\geq 1$, let
\[p_{k}(n,x,y):=\textrm{Card}\left(\mathcal{W}_{n,x,y}/\sim_{k}\right),\]
where \[\mathcal{W}_{n,x,y}:=\left\{w\in\mathcal{F}_{\mathbf{c}}(n) \mid \mathrm{pref}_{k-1}(w)=x,\, \mathrm{suff}_{k-1}(w)=y\right\}.\]
Here $p_{k}(n,x,y)$ denotes the number of $k$-abelian equivalent classes with the prefix $x$ and the suffix $y.$ Then, for every $n\geq 1$, 
\begin{equation}\label{eqn:decomp}
\mathcal{P}^{(k)}_{\mathbf{c}}(n)=\sum_{x,y\in\mathcal{F}_{\mathbf{c}}(k-1)}p_{k}(n,x,y).
\end{equation}
By Theorem \ref{thm:B},
\begin{align}
\nonumber p_{k}(n,x,y) & = \mathrm{Card}\left(\mathcal{W}_{n,x,y}/\sim_{k}\right)\\
& = \mathrm{Card}\left(\mathcal{W}_{n,x,y}/\sim_{1}\right)=\mathrm{Card}\left(\{|w|_{1} \mid w\in\mathcal{W}_{n,x,y}\}\right).
\end{align}
\subsection{Regularity of {the} $2$-abelian complexity of $\mathbf{c}$}
Recall that the Cantor sequence $\mathbf{c}$ is the fixed point of the morphism $\sigma: 0\mapsto 000, 1\mapsto 101$ starting by $1,$ i.e., $\mathbf{c}=\sigma^{\infty}(1).$
\begin{lemma}
For all $i,j\geq 1$, let $d_{j}$ be the number of `$0$' between the $j$-th `$1$' and the $(j+1)$-th `$1$' in $\mathbf{c}$, and let $f(i,j)=j+\sum_{\ell=i}^{i+j-1}d_{\ell}$. Then, for every $j\geq 1$,
\begin{equation}\label{eq:f1}
d_{2j-1}=1 \text{ and } d_{2j}=3d_{j}.
\end{equation}
Moreover, for all $i,j\geq 1$,
\begin{equation}\label{eq:f2}
\left\{
\begin{aligned}
& f(2i,2j)  = 3f(i,j), &  &f(2i,2j+1)  = 3f(i,j+1)-2,\\
& f(2i+1,2j) = 3f(i+1,j), & &f(2i+1,2j+1)  = 3f(i+1,j)+2.
\end{aligned}
\right.
\end{equation}
\end{lemma}
\begin{proof}
While applying $\sigma$ to `$1$' or a block of `$0$'s, we obtain only one block of `$0$'s in both cases. Note that in $\mathbf{c}$, every `$1$' is followed by a block of `$0$'s. Before the $i$-th `$1$', the number of occurrences of `$1$' is $(i-1)$ and there are $(i-1)$ blocks of `$0$'s in $\mathbf{c}$. So, while applying $\sigma$ to $\mathbf{c}$, the $i$-th `$1$' will generate the $(2i-1)$-th block of `$0$'s, which implies $d_{2i-1}=1$. For the same reason, the $i$-th block of `$0$'s will generate the $2i$-th block of `$0$'s. So, $d_{2i}=3d_{i}.$  This proves \eqref{eq:f1}. 

The recurrence relations \eqref{eq:f2} follows directly from \eqref{eq:f1}. We verify the first one as an example:
\[f(2i,2j)=2j+\sum_{\ell=2i}^{2i+2j-1}d_{\ell}=2j+\sum_{\ell=i}^{i+j-1}(d_{2\ell}+d_{2\ell+1})=3j+3\sum_{\ell=i}^{i+j-1}d_{\ell}=3f(i,j).\]
\end{proof}

\begin{proposition}\label{lem:ab2}
$p_{2}(1,0,0)=p_{2}(1,1,1)=1$, $p_{2}(1,0,1)=p_{2}(1,1,0)=0$ and for every $n\geq 2$,
%\begin{subequations}
\begin{subnumcases}{}
p_{2}(n,0,0)   = M_{\mathbf{c}}(n-2)+1, \label{eq:2a}\\
p_{2}(n,1,0)   = p_{2}(n,0,1) = M_{\mathbf{c}}(n-1), \label{eq:2b}\\
p_{2}(n,1,1)   = \begin{cases} 0, & \textrm{ if } n\equiv 0 \mod 2, \\
1, & \textrm{ if } n\equiv 1 \mod 2. \end{cases} \label{eq:2c}
\end{subnumcases}
%\end{subequations}
\end{proposition}
\begin{proof}
The initial values can be showed by enumerating all the factors of length $1$ and $2$. Now, let $n\geq 2$ and suppose $n<3^{i}$ for some $i\geq 1$.

Clearly, for every $w\in\mathcal{W}_{n,0,0}$, $|w|_{1}\leq M_{\mathbf{c}}(n-2)$. So, $p_{2}(n,0,0)\leq M_{\mathbf{c}}(n-2)+1$.  We prove the inverse inequality in the following. For every $0\leq \ell\leq n-1$, let $W_{\ell}=0^{n-\ell}\mathrm{pref}_{\ell}(\sigma^{i}(1))$ that is a factor of $\sigma^{i}(01)$ and hence, a factor of $\mathbf{c}$. Note that $|W_{0}|_{1}=0$ and $|W_{n-2}|_{1}=M_{\mathbf{c}}(n-2)$. Since $|W_{\ell}|_{1}\leq |W_{\ell+1}|_{1}\leq |W_{\ell}|_{1}+1$, we know that $|W_{\ell}|_{1}$ changes continuously from $0$ to $M_{\mathbf{c}}(n-2)$ {while $\ell$ takes values from $0$ to $n-2$}. Therefore, for every $0\leq s\leq M_{\mathbf{c}}(n-2)$, there exists $0\leq \ell\leq n-2$ such that $|W_{\ell}|_{1}=s$. If the last letter of $W_{\ell}$ is $0$, then $W_{\ell}\in \mathcal{W}_{n,0,0}$. Otherwise, $|W_{\ell+1}|_{1}=|W_{\ell}|_{1}=s$ since $11$ is not a factor of $\mathbf{c}$. So, $W_{\ell+1}\in\mathcal{W}_{n,0,0}$. This implies that $p_{2}(n,0,0)\geq M_{\mathbf{c}}(n-2)+1$ which proves \eqref{eq:2a}.

Since for every factor $w$ of $\mathbf{c}$, its reversal $\bar{w}$ is also a factor of $\mathbf{c}$, we have $p_{2}(n,1,0)=p_{2}(n,0,1)$. Then, applying a similar argument on the words $W_{\ell}^{\prime}=\mathrm{suff}_{\ell}(\sigma^{i}(1))0^{n-\ell}$ where $1\leq \ell \leq n-1$, we obtain \eqref{eq:2b}.

(In the rest of the proof, the symbol `$\equiv$', otherwise stated, means equality modulo $2$.)

Now, we prove \eqref{eq:2c} for the case $n\equiv 0$.  We first observe that for every $w\in\mathcal{W}_{n,1,1}$, $|w|\equiv 1$. Since the number of $0$ between two successive $1$ must be $3^{j}$ for some $j\geq 0$ and $3^{j}\equiv 1$, we have $|w|_{0}\equiv |w|_{1}-1$ for every $w\in\mathcal{W}_{n,1,1}$. Therefore, $|w|=|w|_{0}+|w|_{1}\equiv 1$. Hence, $\mathcal{W}_{n,1,1}=\emptyset$ when $n$ is an even number, which implies $p_{2}(n,1,1)=0$ when $n\equiv 0$.

In the following, we will prove \eqref{eq:2c} when $n\equiv 1$. For every $w\in\mathcal{W}_{n,1,1}$, \[n=|w|=|w|_{1}+|w|_{0}=1+f(i,|w|_{1}-1)\] for some $i\geq 1$. (Since if a word occurs in $\mathbf{c}$, then it will occur infinitely many times in $\mathbf{c}$, we can assume $i\geq 3$.) Therefore, we only need to prove that for every $m\geq 1$, there is only one integer $t_{m}\geq 2$ satisfying
\begin{equation}\label{eq:argly}
2m+1=1+f(i,t_{m})
\end{equation}
for some $i\geq 1$. We reason by induction. Since $\mathcal{W}_{3,1,1}=\{101\}$ and $\mathcal{W}_{5,1,1}=\{10001\}$, it follows that \eqref{eq:argly} holds for $m=1$ and $2$.
Assuming that \eqref{eq:argly} {holds} for every $\ell\leq m$, we prove it for $m+1$.  We only give the proof for the case $m=3m^{\prime}$; the other cases follow in a similar way. In this case, by inductive assumptions and \eqref{eq:f2},
\[
2(m+1)+1 = 3(2m^{\prime}+1) =3(1+f(i, t_{m^{\prime}}))=1+f(2i+1,2t_{m^{\prime}}+1),\]
which implies that there is a solution of \eqref{eq:argly} for $m+1$. Now, we prove the uniqueness. Let $t\geq 2$ be a solution of \eqref{eq:argly} for $m+1$. Then,
\begin{equation}
1+f(i,t)=2(m+1)+1=3(2m^{\prime}+1), \label{eq:argly2}
\end{equation}
which implies $f(i,t)\equiv 2~ (\mathrm{mod}~3)$.
According to \eqref{eq:f2}, this happens only if $(i,t)\equiv (1,1)$. Write $i=2i^{\prime}+1$ and $t=2t^{\prime}+1$. Then, by \eqref{eq:f2} and  \eqref{eq:argly2},
\[2m^{\prime}+1=\frac{1+f(i,t)}{3}=1+f(i^{\prime}+1,t^{\prime}).\]
By the inductive assumption, we know that $t^{\prime}$ is the unique solution of \eqref{eq:argly} for $m^{\prime}$. So, the only solution of \eqref{eq:argly} for $m+1$ is $2t_{m^{\prime}}+1$. 
\end{proof}
By Proposition \ref{lem:ab2}, for every $n\geq 2$, we have
\begin{equation*}
\mathcal{P}{^{(2)}_{\mathbf{c}}}(n)= M_{\mathbf{c}}(n-2)+2M_{\mathbf{c}}(n-1)+1+\frac{1+(-1)^{n+1}}{2}.
\end{equation*}
\subsection{Regularity of { the} $k$-abelian complexity of $\mathbf{c}$}
In this part, we prove the regularity of {the} $k$-abelian complexity of the Cantor sequence for every $k\geq 3$. 

Let $\mathcal{F}_{\mathbf{c}}$ denote the set of all factors of $\mathbf{c}.$ For every $u\in\mathcal{F}_{\mathbf{c}}$  and $\ell\geq 1,$ we define
\[\mathrm{Type}(\ell, u):=\big\{j=0,1,\cdots,3^{\ell}-1\mid u=c_{3^{\ell}n+j}\cdots c_{3^{\ell}n+j+|u|-1} \text{ for some }n\geq 0\big\}.\] 
The elements in $\mathrm{Type}(\ell, u)$ are called  \emph{types} of $u$ (with respect to $\ell$).
Clearly, for every $\ell$ and $u\in\mathcal{F}_{\mathbf{c}}$, $\mathrm{Card}(\mathrm{Type}(\ell,u))\geq 1$.

Every type of $u$ gives a decomposition of $u$ in the following sense. For every ${j}\in\mathrm{Type}(\ell,u)$, there is an integer $n\geq 0$ such that
\begin{align}
\notag u & = \left(c_{3^{\ell}n+j}\cdots c_{3^{\ell}(n+1)-1}\right)
\left(c_{3^{\ell}(n+1)} \cdots c_{3^{\ell}(n+h)-1}\right)
\left(c_{3^{\ell}(n+h)}\cdots c_{3^{\ell}n+j+|u|-1}\right)\\
& = \mathrm{suff}_{j_{0}}(\sigma^{\ell}(c_{n}))\, \sigma^{\ell}(c_{n+1}\cdots c_{n+h-1})\, \mathrm{pref}_{j_{1}}(\sigma^{\ell}(c_{n+h})),\label{eq:lword}
\end{align}
where $h=\lfloor \frac{|u|+{j}}{3^{\ell}}\rfloor$, $j_{0}=3^{\ell}-{j}$ and $j_{1}=j+|u|-3^{\ell}h$. The following lemma shows that every non-zero factor of $\mathbf{c}$, which is long enough, occurs in a (relatively) fixed position, i.e., has only one type. By a \emph{non-zero factor} we mean a factor that contains at least one letter `$1$'.
\begin{lemma}\label{lem:1word}
For every integer $\ell\geq 1$ and every non-zero factor $u\in \mathcal{F}_{\mathbf{c}}$ with $|u|> 3^{\ell}$, \[\mathrm{Card}(\mathrm{Type}(\ell, u))=1.\]
\end{lemma}
\begin{proof} We prove by induction on $\ell$. We  first prove the result for $\ell=1$.
Now, we  show that $\mathrm{Card}(\mathrm{Type}(1,u))=1$ for $u\in\mathcal{F}_{\mathbf{c}}(4)$ with $|u|_{1}>0$. We only verify the case $u=0001$ as an example; the rest can be verified in the same way. Suppose $0001=c_{n}c_{n+1}c_{n+2}c_{n+3}$. Since $c_{n+3}=1$, by \eqref{eq:rec}, we have $n\not\equiv 1 ~(\mathrm{mod}~ 3)$. If $n\equiv 2 ~(\mathrm{mod}~ 3)$, then by \eqref{eq:rec}, $0=c_{n+1}=c_{n+3}=1$, which is a contradiction. Thus, $\mathrm{Type}(1,0001)=\{0\}$.

For every non-zero factor $u\in\mathcal{F}_{\mathbf{c}}$ with $|u|> 4$, let $u=xvy$ where $v$ is the the first non-zero factor of length $4$ of $u$. Since $\mathrm{Type}(1,v)+|x|\equiv\mathrm{Type}(1,u) ~(\mathrm{mod}~ 3)$, we have \[\mathrm{Card}(\mathrm{Type}(1,u))=1.\]

Suppose the result holds for $\ell$. We prove it for $\ell+1$. Let $u\in\mathcal{F}_{\mathbf{c}}$ with $|u|> 3^{\ell+1}$ and $i_{0}\in\mathrm{Type}(\ell,u)$. Then,
\[u=c_{3^{\ell}n+i_{0}}\cdots c_{3^{\ell}n+i_{0}+|u|-1}\]
for some $n\geq 0$. By \eqref{eq:lword}, $u$ uniquely determines $i_{0}$, $|u|$ and $c_{n}c_{n+1}\cdots c_{n+h}$ where $h=\lfloor \frac{|u|+i_{0}}{3^{\ell}}\rfloor$. Since $h\geq 3$, $n\equiv i_{1} ~(\mathrm{mod}~ 3)$ where $i_{1}\in \mathrm{Type}(1,c_{n}\cdots c_{n+h})$. Therefore,
\begin{equation}
3^{\ell}n+i_{0}\equiv 3^{\ell}i_{1}+i_{0} \quad(\mathrm{mod}~ 3^{\ell+1}).\label{eq:lem1}
\end{equation}
By the inductive assumptions, $\mathrm{Card}(\mathrm{Type}(1,c_{n}\cdots c_{n+h}))=1$ and $\mathrm{Card}(\mathrm{Type}(\ell,u))=1$. So, by \eqref{eq:lem1}, we have \[\mathrm{Card}(\mathrm{Type}(\ell+1,u))=1.\]
\end{proof}

\begin{lemma}\label{lem:lword}
For every integer $\ell\geq 1$ and every non-zero factor $u\in \mathcal{F}_{\mathbf{c}}$ with $3^{\ell}<|u|\leq 3^{\ell+1}$, \[1\leq \mathrm{Card}(\mathrm{Type}(\ell+1, u))\leq 2.\]
\end{lemma}
\begin{proof}
Let $u\in\mathcal{F}_{\mathbf{c}}$ with $3^{\ell}<|u|\leq 3^{\ell+1}$ and $i_{0}\in\mathrm{Type}(\ell,u)$. Then,
$u=c_{3^{\ell}n+i_{0}}\cdots c_{3^{\ell}n+i_{0}+|u|-1}$
for some $n\geq 0$. By \eqref{eq:lword}, $u$ uniquely determines $i_{0}$, $|u|$ and $c_{n}c_{n+1}\cdots c_{n+h}=:v,$ where $h=\lfloor \frac{|u|+i_{0}}{3^{\ell}}\rfloor$. Note that $v$ is a non-zero factor. Write $q(v):=\max\{j\mid 0^{j} \text{ is a prefix of }v\}$. Then $c_{n+q(v)}=1$, which implies $n+q(v)\not\equiv 1 ~(\mathrm{mod}~ 3)$ by \eqref{eq:rec}. So,
\begin{equation}\label{eq:lem2}
3^{\ell}n+i_{0}\equiv -3^{\ell}q(v)+i_{0} \text{ or } 3^{\ell}(2-q(v))+i_{0} ~(\mathrm{mod}~ 3^{\ell+1}).
\end{equation}
The result follows from Lemma \ref{lem:1word} and the above formula.
\end{proof}

In the rest of this section, let $i$  be the integer satisfying \[3^i+1<k\leq 3^{i+1}+1.\]

To study the regularity of $\{p_{k}(n,x,y)\}_{n\geq 1}$ for $x,y\in\mathcal{F}_{\mathbf{c}}(k-1)$, 
our idea is the following. We first give the upper bound of $p_{k}(n,\cdot,\cdot)$ by using $M_{\mathbf{c}}(\cdot)$, which is a $3$-regular sequence according to Corollary \ref{max:min}. Then, by constructing sufficiently many words which belong to different $k$-abelian equivalence classes, we show that the upper bound can be reached. Therefore, the regularity of $\{p_{k}(n,x,y)\}_{n\geq 1}$ follows from the regularity of $\{M_{\mathbf{c}}(n)\}_{n\geq 1}$.

The following lemma contributes to the construction of words that belong to different $k$-abelian equivalence classes. 
\begin{lemma}\label{lem:w}
Let $\alpha\in\{0,1\}$. For every $\ell\geq 1$ and every $h=1,2,\cdots, M_{\mathbf{c}}(\ell)$, there is a word $W_{h}\in\mathcal{F}_{\mathbf{c}}(\ell+3)$ such that $|W_{h}|_{1}=h$ and $W_h=00U_{h}\alpha$, where $U_h\in\mathcal{F}_{\mathbf{c}}(\ell)$.  
\end{lemma} 
\begin{proof}
For all $j=0,1,\cdots,\ell+1$, let \[W_{j}=0^{\ell+3-j}\mathrm{pref}_{j}(\sigma^{s}(1))\in\mathcal{F}_{c}(\ell+3),\] where $s\in\mathbb{N}$ satisfying $3^{s}>\ell+1$. Since $|W_{j}|_{1}\leq |W_{j+1}|_{1}\leq |W_{j}|_{1}+1$ and $|W_{\ell}|_{1}=M_{\mathbf{c}}(\ell)$, we know that $|W_{j}|_{1}$ changes from 0 to $M_{\mathbf{c}}(\ell)$ continuously while $j$ takes values from $0$ to $\ell$. So, for every $h=1,\cdots, M_{\mathbf{c}}(\ell)$, there is a $j_{h} (\leq \ell)$ such that $|W_{j_{h}}|_{1}=h$. Moreover, we can require that the last letter of $W_{j_{h}}$ is $0$. Otherwise, $1$ is the last letter of $W_{j_{h}}$. Then, $W_{j_{h}+1}$ ends with $0$ and $|W_{j_{h}+1}|_{1}=|W_{j_{h}}|_{1}$.

There also is a $j^{\prime}_{h}$ such that $|W_{j^{\prime}_{h}}|_{1}=h$, of which the last letter is $1$. Otherwise, $0$ is the last letter of $W_{j^{\prime}_{h}}$. Let $m_h:=\max\{q\mid 0^{q}~\text{is a suffix of}~W_{j^{\prime}_{h}}\}$. Since  $|W_{j^{\prime}_{h}}|_{1}=h\geq 1$, we always have $m_{h}< j^{\prime}_{h}$. Then, $W_{j^{\prime}_{h}-m_h}$ ends with $1$ and $|W_{j^{\prime}_{h}-m_h}|_{1}=|W_{j^{\prime}_{h}}|_{1}$. If $m_{h}> j^{\prime}_{h}$,
\end{proof}

Now, we shall show the regularity of $\{p_{k}(n,x,y)\}_{n\geq 1}$ for all $x,y\in\mathcal{F}_{\mathbf{c}}(k-1)$. 

\begin{lemma}\label{lem:0k-0k}
$\{p_{k}(n,0^{k-1},0^{k-1})\}_{n\geq 1}$ is a $3$-regular sequence.
\end{lemma}
\begin{proof}
Without loss of generality, we can assume that $n\geq 2\cdot 3^{i+1}+ 2k-2$, since changing finite terms of a sequence does not change its regularity. Noticing that $3^{i}< k-1\leq 3^{i+1}$, the occurrence of  each $w\in\mathcal{W}_{n,0^{k-1},0^{k-1}}$ in $\mathbf{c}$ must be one of the four forms in Figure \ref{fig:1}. 

\begin{figure}[htbp]
\scalebox{0.8}{
\begin{tabular}{cc}
\begin{minipage}{0.5\textwidth}
\begin{tikzpicture}
%\draw[step=1cm,help lines] (-.5,-1) grid (4.5,1);
\filldraw[fill=blue!30,thick] (0,0) rectangle (4.2,0.5);
\filldraw[fill=red!30,thick] (1.25,0) rectangle (3,0.5);
\node at (0.65,0.25) {{$\sigma^{i+1}(0)$}};
\node at (2.1,0.25) {{$\sigma^{i+1}(u)$}};
\node at (3.6,0.25) {{$\sigma^{i+1}(0)$}};
\filldraw[fill=brown!30,thick] (0.3,0) rectangle (4,-0.5);
\filldraw[fill=green!30,thick] (1.1,0) rectangle (3.2,-0.5);
\node at (0.75,-0.25) {{$0^{k-1}$}};
\node at (2.1,-0.25) {{$w^{\prime}$}};
\node at (3.6,-0.25) {{$0^{k-1}$}};
\node at (2,-0.8) { Form $1$};
\draw[<-,thick] (4.3,0.25) -- (5,0.25) node [right] { A fragment of $\mathbf{c}$};
\draw[<-,thick] (4.3,-0.25) -- (5,-0.25) node [right] { $w=0^{k-1}w^{\prime}0^{k-1}$};
\end{tikzpicture}
\end{minipage} &
\begin{minipage}{0.5\textwidth}
\begin{tikzpicture}
%\draw[step=1cm,help lines] (-.5,-1) grid (4.5,1);
\filldraw[fill=blue!30,thick] (0,0) rectangle (4.2,0.5);
\filldraw[fill=red!30,thick] (1.25,0) rectangle (3,0.5);
\filldraw[fill=blue!30,thick] (4.2,0) rectangle (5.4,0.5);
\node at (0.65,0.25) {{$\sigma^{i+1}(0)$}};
\node at (2.1,0.25) {{$\sigma^{i+1}(u)$}};
\node at (3.6,0.25) {{$\sigma^{i+1}(0)$}};
\node at (4.8,0.25) {{$\sigma^{i+1}(0)$}};
\filldraw[fill=brown!30,thick] (0.3,0) rectangle (4.8,-0.5);
\filldraw[fill=green!30,thick] (1.1,0) rectangle (3.8,-0.5);
\node at (0.75,-0.25) {{$0^{k-1}$}};
\node at (2.6,-0.25) {{$w^{\prime}$}};
\node at (4.4,-0.25) {{$0^{k-1}$}};
\node at (2.6,-0.8) { Form $2$};
\end{tikzpicture}
\end{minipage} \\
\begin{minipage}{0.5\textwidth}
\begin{tikzpicture}
%\draw[step=1cm,help lines] (-.5,-1) grid (4.5,1);
\filldraw[fill=blue!30,thick] (0,0) rectangle (4.2,0.5);
\filldraw[fill=red!30,thick] (1.25,0) rectangle (3,0.5);
\filldraw[fill=blue!30,thick] (-1.2,0) rectangle (0,0.5);
\node at (-0.6,0.25) {{$\sigma^{i+1}(0)$}};
\node at (0.65,0.25) {{$\sigma^{i+1}(0)$}};
\node at (2.1,0.25) {{$\sigma^{i+1}(u)$}};
\node at (3.6,0.25) {{$\sigma^{i+1}(0)$}};
\filldraw[fill=brown!30,thick] (-0.6,0) rectangle (4,-0.5);
\filldraw[fill=green!30,thick] (0.3,0) rectangle (3.2,-0.5);
\node at (-.1,-0.25) {{$0^{k-1}$}};
\node at (1.7,-0.25) {{$w^{\prime}$}};
\node at (3.6,-0.25) {{$0^{k-1}$}};
\node at (0.8,-0.8) { Form $3$};
\end{tikzpicture}
\end{minipage} &
\begin{minipage}{0.5\textwidth}
\begin{tikzpicture}
%\draw[step=1cm,help lines] (-.5,-1) grid (4.5,1);
\filldraw[fill=blue!30,thick] (0,0) rectangle (4.2,0.5);
\filldraw[fill=red!30,thick] (1.25,0) rectangle (3,0.5);
\filldraw[fill=blue!30,thick] (-1.2,0) rectangle (0,0.5);
\filldraw[fill=blue!30,thick] (4.2,0) rectangle (5.4,0.5);
\node at (-0.6,0.25) {{$\sigma^{i+1}(0)$}};
\node at (0.65,0.25) {{$\sigma^{i+1}(0)$}};
\node at (2.1,0.25) {{$\sigma^{i+1}(u)$}};
\node at (3.6,0.25) {{$\sigma^{i+1}(0)$}};
\node at (4.8,0.25) {{$\sigma^{i+1}(0)$}};
\filldraw[fill=brown!30,thick] (-.6,0) rectangle (4.8,-0.5);
\filldraw[fill=green!30,thick] (.3,0) rectangle (3.8,-0.5);
\node at (-.1,-0.25) {{$0^{k-1}$}};
\node at (2.1,-0.25) {{$w^{\prime}$}};
\node at (4.4,-0.25) {{$0^{k-1}$}};
\node at (1.4,-0.8) { Form $4$};
\end{tikzpicture}
\end{minipage}
\end{tabular}
}
\caption{}\label{fig:1}
\end{figure}

In all the four forms, we have $|w|_{1}=2^{i+1}|u|_{1}$ and
$|u|=\ell$ or $\ell-1$, which implies
\begin{equation}\label{eq:0k}
p_{k}(n,0^{k-1},0^{k-1})\leq M_{\mathbf{c}}(\ell)+1,
\end{equation}
where $\ell=\left\lfloor \frac{n-2k+2}{3^{i+1}}\right\rfloor$.
Next, we prove the inverse of \eqref{eq:0k}. That is
\begin{equation}\label{eq:0k2}
p_{k}(n,0^{k-1},0^{k-1})\geq M_{\mathbf{c}}(\ell)+1.
\end{equation}
Applying Lemma \ref{lem:w} for the above $\ell$ and $\alpha=0$, we have 
\[W_{h}=00U_{h}0\in \mathcal{F}_{\mathbf{c}}(\ell+3) \text{ with } |W_{h}|=h\]
for all $h=1,2,\cdots, M_{\mathbf{c}}(\ell)$. 
Set $t:=n-3^{i+1}\ell-k+1$. Then, $k-1\leq t <k-1+3^{i+1}$. Therefore,
\[0^{t}\sigma^{i+1}(U_{h})0^{k-1}\in\mathcal{W}_{n,0^{k-1},0^{k-1}} \text{ and } |0^{t}\sigma^{i+1}(U_{h})0^{k-1}|_{1}=2^{i+1}h\]
for every $h=1,\cdots,M_{\mathbf{c}}(\ell)$. Noting also that $0^{n}\in\mathcal{W}_{n,0^{k-1},0^{k-1}}$, the inequality \eqref{eq:0k2} holds. The result then follows from \eqref{eq:0k}, \eqref{eq:0k2} and Corollary \ref{max:min}.
\end{proof}

For every non-zero factor $v\in \mathcal{F}_{\mathbf{c}}(k-1)$, let $z_{v}:=\max\{p \mid 0^{p} \text{ is a suffix of }v\}$ and 
\[\tilde{L}_{v}:=\{q~\text{mod}~3^{i+1} \mid c_{q-(k-2-z_{v})}\cdots c_{q-1}c_{q}c_{q+1}\cdots c_{q+z_{v}}=v\},\]
where $c_{q}$ is the last $1$ in $v$. Then, it follows from Lemma \ref{lem:lword} that $1\leq \text{Card}(\tilde{L}_{v})\leq 2$. Moreover, if $\tilde{L}_{v}=\{q_{1},q_{2}\}$, where $0\leq q_{1}<q_{2}\leq 3^{i+1}-1$, then by \eqref{eq:lem2}, we have $q_{2}=q_{1}+2\cdot 3^{i}$.

For a word $w=w_{0}w_{1}\cdots w_{n-1}\in \mathcal{A}^{n}$, the reversal of $w$ is defined to be $\bar{w}=w_{n-1}\cdots w_{1}w_{0}$.  When $w=uv$, we write $wv^{-1}:=u$ and $u^{-1}w:=v$ by convention.

\begin{lemma}\label{lem:0k-x}
For all non-zero factors {$x,y\in\mathcal{F}_{\mathbf{c}}(k-1)$}, two sequences $\{p_{k}(n,0^{k-1},y)\}_{n\geq 1}$ and $\{p_{k}(n,x,0^{k-1})\}_{n\geq 1}$ are both $3$-regular sequences.
\end{lemma}
\begin{proof}
For every $x\in\mathcal{F}_{\mathbf{c}}$, its reversal $\bar{x}\in\mathcal{F}_{\mathbf{c}}$, since $x$ is a factor of $\sigma^{{m}}(1)$ for some ${m}\geq 1$ and $\overline{\sigma^{{m}}(1)}=\sigma^{{m}}(1)$. So, $p_{k}(n,x,0^{k-1})=p_{k}(n,0^{k-1},\bar{x})$ for every $n\geq 1$. Thus, we only need to verify the regularity of {$\{p_{k}(n,0^{k-1},y)\}_{n\geq 1}$} for every non-zero factor {$y\in\mathcal{F}_{\mathbf{c}}(k-1)$}.

Since changing finite terms of a sequence does not change its regularity,  we can assume that $n\geq 2\cdot 3^{i+1}+ 2k-2$. Recall that $3^{i}< k-1\leq 3^{i+1}$. Each occurrence of every $w\in\mathcal{W}_{n,0^{k-1},y}$ in $\mathbf{c}$ must be one of the six forms in Figure \ref{fig:2}. In all the six forms, for every $o_y \in \tilde{L}_y$, we have 
\begin{equation}\label{eq:wno}
|w|_{1}=2^{i+1}|\tilde{u}|_{1}-|\mathrm{suff}_{3^{i+1}-o_y-1}(\sigma^{i+1}(1))|_1:=n_{o_{y}}
\end{equation}
and
$|\tilde{u}|=\ell(o_{y})$ or $\ell(o_{y})+1$, where 
\[\ell(o_{y})=\left\lfloor \frac{n-k-o_y-z_y}{3^{i+1}}\right\rfloor 
\text{ and }
\tilde{u} = \begin{cases}
u01, & \text{if }w \text{ is of Form }5 \text{ or }6,\\
u1, &  \text{otherwise}.\\
\end{cases} \]
\vspace{-3ex}
\begin{figure}[htbp]
\scalebox{0.8}{
\begin{tabular}{cc}
\begin{minipage}{0.5\textwidth}
\begin{tikzpicture}
%\draw[step=1cm,help lines] (-.5,-1) grid (4.5,1);
\filldraw[fill=blue!30,thick] (0,0) rectangle (4.2,0.5);
\filldraw[fill=red!30,thick] (1.25,0) rectangle (3,0.5);
\node at (0.65,0.25) {{$\sigma^{i+1}(0)$}};
\node at (2.1,0.25) {{$\sigma^{i+1}(u)$}};
\node at (3.6,0.25) {{$\sigma^{i+1}(1)$}};
\filldraw[fill=brown!30,thick] (0.3,0) rectangle (4,-0.5);
\filldraw[fill=green!30,thick] (1.1,0) rectangle (3.2,-0.5);
\node at (0.75,-0.25) {{$0^{k-1}$}};
\node at (2.1,-0.25) {{$w^{\prime}$}};
\node at (3.6,-0.25) {{$y$}};
\node at (2,-.9) { Form $1$};
\draw[<-,thick] (4.3,0.25) -- (5,0.25) node [right] { A fragment of $\mathbf{c}$};
\draw[<-,thick] (4.3,-0.25) -- (5,-0.25) node [right] { $w=0^{k-1}w^{\prime}y$};
\draw [|<->|,thick] (3.2,-0.7) -- (4,-0.7) node [below,midway] {$ k-1$};
\draw [|<->|,thick] (3,0.7) -- (4,0.7) node [above,midway] {$o_y+z_y$};
\end{tikzpicture}
\end{minipage} \vspace{1ex}& 
\begin{minipage}{0.5\textwidth}
\begin{tikzpicture}
%\draw[step=1cm,help lines] (-.5,-1) grid (4.5,1);
\filldraw[fill=blue!30,thick] (0,0) rectangle (4.2,0.5);
\filldraw[fill=red!30,thick] (1.25,0) rectangle (3,0.5);
\filldraw[fill=blue!30,thick] (-1.2,0) rectangle (0,0.5);
\node at (-0.6,0.25) {{$\sigma^{i+1}(0)$}};
\node at (0.65,0.25) {{$\sigma^{i+1}(0)$}};
\node at (2.1,0.25) {{$\sigma^{i+1}(u)$}};
\node at (3.6,0.25) {{$\sigma^{i+1}(1)$}};
\filldraw[fill=brown!30,thick] (-0.6,0) rectangle (4,-0.5);
\filldraw[fill=green!30,thick] (0.3,0) rectangle (3.2,-0.5);
\node at (-.1,-0.25) {{$0^{k-1}$}};
\node at (1.7,-0.25) {{$w^{\prime}$}};
\node at (3.6,-0.25) {{$y$}};
\node at (1.4,-0.8) { Form $2$};
\end{tikzpicture}\vspace{-4ex}
\end{minipage}\\
\begin{minipage}{0.5\textwidth}
\begin{tikzpicture}
%\draw[step=1cm,help lines] (-.5,-1) grid (4.5,1);
\filldraw[fill=blue!30,thick] (0,0) rectangle (4.2,0.5);
\filldraw[fill=red!30,thick] (1.25,0) rectangle (3,0.5);
\filldraw[fill=blue!30,thick] (4.2,0) rectangle (5.4,0.5);
\node at (0.65,0.25) {{$\sigma^{i+1}(0)$}};
\node at (2.1,0.25) {{$\sigma^{i+1}(u)$}};
\node at (3.6,0.25) {{$\sigma^{i+1}(1)$}};
\node at (4.8,0.25) {{$\sigma^{i+1}(0)$}};
\filldraw[fill=brown!30,thick] (0.3,0) rectangle (4.8,-0.5);
\filldraw[fill=green!30,thick] (1.1,0) rectangle (3.8,-0.5);
\node at (0.75,-0.25) {{$0^{k-1}$}};
\node at (2.6,-0.25) {{$w^{\prime}$}};
\node at (4.4,-0.25) {{$y$}};
\node at (2.6,-0.8) { Form $3$};
\end{tikzpicture}
\end{minipage}\vspace{1ex}&
\begin{minipage}{0.5\textwidth}
\begin{tikzpicture}
%\draw[step=1cm,help lines] (-.5,-1) grid (4.5,1);
\filldraw[fill=blue!30,thick] (0,0) rectangle (4.2,0.5);
\filldraw[fill=red!30,thick] (1.25,0) rectangle (3,0.5);
\filldraw[fill=blue!30,thick] (-1.2,0) rectangle (0,0.5);
\filldraw[fill=blue!30,thick] (4.2,0) rectangle (5.4,0.5);
\node at (-0.6,0.25) {{$\sigma^{i+1}(0)$}};
\node at (0.65,0.25) {{$\sigma^{i+1}(0)$}};
\node at (2.1,0.25) {{$\sigma^{i+1}(u)$}};
\node at (3.6,0.25) {{$\sigma^{i+1}(1)$}};
\node at (4.8,0.25) {{$\sigma^{i+1}(0)$}};
\filldraw[fill=brown!30,thick] (-.6,0) rectangle (4.8,-0.5);
\filldraw[fill=green!30,thick] (.3,0) rectangle (3.8,-0.5);
\node at (-.1,-0.25) {{$0^{k-1}$}};
\node at (2.1,-0.25) {{$w^{\prime}$}};
\node at (4.4,-0.25) {{$y$}};
\node at (1.4,-0.8) { Form $4$};
\end{tikzpicture}
\end{minipage}\vspace{1ex} \\
\begin{minipage}{0.5\textwidth}
\begin{tikzpicture}
%\draw[step=1cm,help lines] (-.5,-1) grid (4.5,1);
\filldraw[fill=blue!30,thick] (0,0) rectangle (4.2,0.5);
\filldraw[fill=red!30,thick] (1.25,0) rectangle (3,0.5);
\filldraw[fill=blue!30,thick] (4.2,0) rectangle (5.4,0.5);
\node at (0.65,0.25) {{$\sigma^{i+1}(0)$}};
\node at (2.1,0.25) {{$\sigma^{i+1}(u)$}};
\node at (3.6,0.25) {{$\sigma^{i+1}(0)$}};
\node at (4.8,0.25) {{$\sigma^{i+1}(1)$}};
\filldraw[fill=brown!30,thick] (0.3,0) rectangle (4.8,-0.5);
\filldraw[fill=green!30,thick] (1.1,0) rectangle (3.8,-0.5);
\node at (0.75,-0.25) {{$0^{k-1}$}};
\node at (2.6,-0.25) {{$w^{\prime}$}};
\node at (4.4,-0.25) {{$y$}};
\node at (2.6,-0.8) { Form $5$};
\end{tikzpicture}
\end{minipage}\vspace{1ex} &
\begin{minipage}{0.5\textwidth}
\begin{tikzpicture}
%\draw[step=1cm,help lines] (-.5,-1) grid (4.5,1);
\filldraw[fill=blue!30,thick] (0,0) rectangle (4.2,0.5);
\filldraw[fill=red!30,thick] (1.25,0) rectangle (3,0.5);
\filldraw[fill=blue!30,thick] (-1.2,0) rectangle (0,0.5);
\filldraw[fill=blue!30,thick] (4.2,0) rectangle (5.4,0.5);
\node at (-0.6,0.25) {{$\sigma^{i+1}(0)$}};
\node at (0.65,0.25) {{$\sigma^{i+1}(0)$}};
\node at (2.1,0.25) {{$\sigma^{i+1}(u)$}};
\node at (3.6,0.25) {{$\sigma^{i+1}(0)$}};
\node at (4.8,0.25) {{$\sigma^{i+1}(1)$}};
\filldraw[fill=brown!30,thick] (-.6,0) rectangle (4.8,-0.5);
\filldraw[fill=green!30,thick] (.3,0) rectangle (3.8,-0.5);
\node at (-.1,-0.25) {{$0^{k-1}$}};
\node at (2.1,-0.25) {{$w^{\prime}$}};
\node at (4.4,-0.25) {{$y$}};
\node at (1.4,-0.8) {Form $6$};
\end{tikzpicture}
\end{minipage}
\end{tabular}
}
\caption{}\label{fig:2}
\end{figure}

%We need to consider two cases: $\text{Card}(\tilde{L}_y)=1$ or $2$.
When $\text{Card}(\tilde{L}_y)=1$,  write $\tilde{L}_y=\{o_y\}$. By \eqref{eq:wno}, we have
\begin{equation}\label{eq:0k-x}
p_{k}(n,0^{k-1},y)\leq M_{\mathbf{c}}\left(\ell(o_{y})+1\right).
\end{equation}
On the other hand, applying Lemma \ref{lem:w} for $\ell(o_{y})$ and $\alpha=1$, we have 
\[W_{h}=00U_{h}1\in\mathcal{F}_{\mathbf{c}}(\ell+3) \text{ with }|W_{h}|_{1}=h\] 
for all $h=1,\cdots, M_{\mathbf{c}}(\ell+1)$.  Set $t:=n-3^{i+1}\ell-o_y-z_y-1$; so $k-1\leq t <k-1+3^{i+1}$. Therefore,
\[V_{o_{y}}:=0^{t}\sigma^{i+1}(U_{h})\mathrm{pref}_{o_y+1}(\sigma^{i+1}(1))0^{z_{y}}\in\mathcal{W}_{n,0^{k-1},y} \]
and
\[ |V_{o_{y}}|_{1}=2^{i+1}h-|\mathrm{suff}_{3^{i+1}-o_y-1}(\sigma^{i+1}(1))|_1 \]
for all $h=1,\cdots,M_{\mathbf{c}}(\ell+1)$. This implies
that $p_{k}(n,0^{k-1},y)\geq M_{\mathbf{c}}\left(\ell(o_{y})+1\right)$. The previous inequality, \eqref{eq:0k-x} and Corollary \ref{max:min} give the result in the case $\text{Card}(\tilde{L}_y)=1$.

Now suppose $\text{Card}(\tilde{L}_y)=2$ and set $\tilde{L}_y=\{o_y, o_{y}^{\prime}:=o_y+2\cdot 3^i\}$ with $0 \leq o_y \leq 3^{i}-1$. From \eqref{eq:wno}, we know that $n_{o_{y}^{\prime}}\equiv n_{o_{y}}+2^{i}\mod 2^{i+1}$. Therefore, 
\begin{equation}
p_{k}(n,0^{k-1},y) \leq  M_{\mathbf{c}}\left( \ell(o_{y}) +1\right)+M_{\mathbf{c}}\left( \ell(o_{y}^{\prime})+1\right).
\label{eq:0k-x2}
\end{equation}
For every $q\in\tilde{L}_{y}$, applying Lemma \ref{lem:w} for $\ell(q)$ and $\alpha=1$, we have  
\[W_{h,q}=00U_{h,q}1\in\mathcal{F}_{\mathbf{c}}(\ell+3) \text{ with }|W_{h,q}|_{1}=h\] 
for every $h=1,\cdots,M_{\mathbf{c}}(\ell_1+1)$. Set 
$t(q):=n-3^{i+1}\ell(q)-q-z_y-1$; so $k-1\leq t(q) <k-1+3^{i+1}$. 
Therefore, for every $q\in\tilde{L}_{y}$,
\[V_{q}:=0^{t(q)}\sigma^{i+1}(U_{h,q})\mathrm{pref}_{q+1}(\sigma^{i+1}(1))0^{z_{y}}\in \mathcal{W}_{n,0^{k-1},y}\] 
and
\[|V_{q}|_{1}=2^{i+1}h-|\mathrm{suff}_{3^{i+1}-q-1}(\sigma^{i+1}(1))|_1\]
for all $h=1,\cdots,M_{\mathbf{c}}(\ell_1+1)$. Since $|V_{o_{y}}|_{1}\equiv |V_{o_{y}^{\prime}}|_{1} -2^{i}$ (mod $2^{i+1}$), $V_{o_{y}}$ and $V_{o_{y}^{\prime}}$ belongs to different $k$-abelian equivalence classes. Therefore, 
\begin{equation}
p_{k}(n,0^{k-1},y) \geq  M_{\mathbf{c}}\left( \ell(o_{y}) +1\right)+M_{\mathbf{c}}\left( \ell(o_{y}^{\prime})+1\right).
\label{eq:0k-x2v}
\end{equation}
Combining \eqref{eq:0k-x2}, \eqref{eq:0k-x2v} and Corollary \ref{max:min}, the result follows.
\end{proof}

\begin{lemma}\label{lem:x-y}
For two non-zero factors {$x,y\in\mathcal{F}_{\mathbf{c}}(k-1)$}, $\{p_{k}(n,x,y)\}_{n\geq 1}$ is ultimately periodic.
\end{lemma}
\begin{proof}
Without loss of generality, we can assume that $n\geq 2\cdot 3^{i+1}+2k-2$ since changing finite terms of a sequence does not change its regularity. Noticing that $3^{i}< k-1\leq 3^{i+1}$, for every pair of factors $x,y$ of length $k-1,$ the occurrence of  each $w\in\mathcal{W}_{n,x,y}$ in $\mathbf{c}$ must be one of the nine forms in Figure \ref{fig:3}.

\begin{figure}[htbp]
\scalebox{0.8}{
\begin{tabular}{cc}
\begin{minipage}{0.52\textwidth}
\begin{tikzpicture}
%\draw[step=1cm,help lines] (-.5,-1) grid (4.5,1);
\filldraw[fill=blue!30,thick] (0,0) rectangle (4.2,0.5);
\filldraw[fill=red!30,thick] (1.25,0) rectangle (3,0.5);
\filldraw[fill=blue!30,thick] (-1.2,0) rectangle (0,0.5);
\node at (-0.6,0.25) {{$\sigma^{i+1}(1)$}};
\node at (0.65,0.25) {{$\sigma^{i+1}(0)$}};
\node at (2.1,0.25) {{$\sigma^{i+1}(u)$}};
\node at (3.6,0.25) {{$\sigma^{i+1}(1)$}};
\filldraw[fill=brown!30,thick] (-0.6,0) rectangle (4,-0.5);
\filldraw[fill=green!30,thick] (0.3,0) rectangle (3.2,-0.5);
\node at (-.1,-0.25) {{$x$}};
\node at (1.7,-0.25) {{$w^{\prime}$}};
\node at (3.6,-0.25) {{$y$}};
\node at (1.55,-0.9) { Form $1$};
\draw[<-,thick] (4.3,0.25) -- (5,0.25) node [right] { A fragment of $\mathbf{c}$};
\draw[<-,thick] (4.3,-0.25) -- (5,-0.25) node [right] { $w=xw^{\prime}y$};

\draw [|<->|,thick] (3.2,-0.7) -- (4,-0.7) node [below,midway] {$ k-1$};
\draw [|<->|,thick] (3,0.7) -- (4,0.7) node [above,midway] {$ o_y+z_y$};
\draw [|<->|,thick] (-0.6,-0.7) -- (0.3,-0.7) node [below,midway] {$ k-1$};
\draw [|<->|,thick] (-1.2,0.7) -- (0.3,0.7) node [above,midway] {$o_x+z_x$};
\end{tikzpicture}
\end{minipage}\vspace{1ex} &
\begin{minipage}{0.4\textwidth}
\begin{tikzpicture}
%\draw[step=1cm,help lines] (-.5,-1) grid (4.5,1);
\filldraw[fill=blue!30,thick] (0,0) rectangle (4.2,0.5);
\filldraw[fill=red!30,thick] (1.25,0) rectangle (3,0.5);
\filldraw[fill=blue!30,thick] (-1.2,0) rectangle (0,0.5);
\node at (-0.6,0.25) {{$\sigma^{i+1}(0)$}};
\node at (0.65,0.25) {{$\sigma^{i+1}(1)$}};
\node at (2.1,0.25) {{$\sigma^{i+1}(u)$}};
\node at (3.6,0.25) {{$\sigma^{i+1}(1)$}};
\filldraw[fill=brown!30,thick] (-0.6,0) rectangle (4,-0.5);
\filldraw[fill=green!30,thick] (0.3,0) rectangle (3.2,-0.5);
\node at (-.1,-0.25) {{$x$}};
\node at (1.7,-0.25) {{$w^{\prime}$}};
\node at (3.6,-0.25) {{$y$}};
\node at (0.8,-0.8) { Form $2$};
\end{tikzpicture}
\vspace{-4ex}
\end{minipage}\vspace{1ex} \\

\begin{minipage}{0.5\textwidth}
\begin{tikzpicture}
%\draw[step=1cm,help lines] (-.5,-1) grid (4.5,1);
\filldraw[fill=blue!30,thick] (0,0) rectangle (4.2,0.5);
\filldraw[fill=red!30,thick] (1.25,0) rectangle (3,0.5);
\filldraw[fill=blue!30,thick] (-1.2,0) rectangle (0,0.5);
\filldraw[fill=blue!30,thick] (4.2,0) rectangle (5.4,0.5);
\node at (-0.6,0.25) {{$\sigma^{i+1}(1)$}};
\node at (0.65,0.25) {{$\sigma^{i+1}(0)$}};
\node at (2.1,0.25) {{$\sigma^{i+1}(u)$}};
\node at (3.6,0.25) {{$\sigma^{i+1}(0)$}};
\node at (4.8,0.25) {{$\sigma^{i+1}(1)$}};
\filldraw[fill=brown!30,thick] (-.6,0) rectangle (4.8,-0.5);
\filldraw[fill=green!30,thick] (.3,0) rectangle (3.8,-0.5);
\node at (-.1,-0.25) {{$x$}};
\node at (2.1,-0.25) {{$w^{\prime}$}};
\node at (4.4,-0.25) {{$y$}};
\node at (1.4,-0.8) { Form $3$};
\end{tikzpicture}
\end{minipage}\vspace{1ex} &
\begin{minipage}{0.5\textwidth}
\begin{tikzpicture}
%\draw[step=1cm,help lines] (-.5,-1) grid (4.5,1);
\filldraw[fill=blue!30,thick] (0,0) rectangle (4.2,0.5);
\filldraw[fill=red!30,thick] (1.25,0) rectangle (3,0.5);
\filldraw[fill=blue!30,thick] (-1.2,0) rectangle (0,0.5);
\filldraw[fill=blue!30,thick] (4.2,0) rectangle (5.4,0.5);
\node at (-0.6,0.25) {{$\sigma^{i+1}(0)$}};
\node at (0.65,0.25) {{$\sigma^{i+1}(1)$}};
\node at (2.1,0.25) {{$\sigma^{i+1}(u)$}};
\node at (3.6,0.25) {{$\sigma^{i+1}(1)$}};
\node at (4.8,0.25) {{$\sigma^{i+1}(0)$}};
\filldraw[fill=brown!30,thick] (-.6,0) rectangle (4.8,-0.5);
\filldraw[fill=green!30,thick] (.3,0) rectangle (3.8,-0.5);
\node at (-.1,-0.25) {{$x$}};
\node at (2.1,-0.25) {{$w^{\prime}$}};
\node at (4.4,-0.25) {{$y$}};
\node at (1.4,-0.8) { Form $4$};
\end{tikzpicture}
\end{minipage}\vspace{1ex} \\
\begin{minipage}{0.5\textwidth}
\begin{tikzpicture}
%\draw[step=1cm,help lines] (-.5,-1) grid (4.5,1);
\filldraw[fill=blue!30,thick] (0,0) rectangle (4.2,0.5);
\filldraw[fill=red!30,thick] (1.25,0) rectangle (3,0.5);
\filldraw[fill=blue!30,thick] (-1.2,0) rectangle (0,0.5);
\filldraw[fill=blue!30,thick] (4.2,0) rectangle (5.4,0.5);
\node at (-0.6,0.25) {{$\sigma^{i+1}(1)$}};
\node at (0.65,0.25) {{$\sigma^{i+1}(0)$}};
\node at (2.1,0.25) {{$\sigma^{i+1}(u)$}};
\node at (3.6,0.25) {{$\sigma^{i+1}(1)$}};
\node at (4.8,0.25) {{$\sigma^{i+1}(0)$}};
\filldraw[fill=brown!30,thick] (-.6,0) rectangle (4.8,-0.5);
\filldraw[fill=green!30,thick] (.3,0) rectangle (3.8,-0.5);
\node at (-.1,-0.25) {{$x$}};
\node at (2.1,-0.25) {{$w^{\prime}$}};
\node at (4.4,-0.25) {{$y$}};
\node at (1.4,-0.8) { Form $5$};
\end{tikzpicture}
\end{minipage}\vspace{1ex}  &

\begin{minipage}{0.5\textwidth}
\begin{tikzpicture}
%\draw[step=1cm,help lines] (-.5,-1) grid (4.5,1);
\filldraw[fill=blue!30,thick] (0,0) rectangle (4.2,0.5);
\filldraw[fill=red!30,thick] (1.25,0) rectangle (3,0.5);
\filldraw[fill=blue!30,thick] (-1.2,0) rectangle (0,0.5);
\filldraw[fill=blue!30,thick] (4.2,0) rectangle (5.4,0.5);
\node at (-0.6,0.25) {{$\sigma^{i+1}(0)$}};
\node at (0.65,0.25) {{$\sigma^{i+1}(1)$}};
\node at (2.1,0.25) {{$\sigma^{i+1}(u)$}};
\node at (3.6,0.25) {{$\sigma^{i+1}(0)$}};
\node at (4.8,0.25) {{$\sigma^{i+1}(1)$}};
\filldraw[fill=brown!30,thick] (-.6,0) rectangle (4.8,-0.5);
\filldraw[fill=green!30,thick] (.3,0) rectangle (3.8,-0.5);
\node at (-.1,-0.25) {{$x$}};
\node at (2.1,-0.25) {{$w^{\prime}$}};
\node at (4.4,-0.25) {{$y$}};
\node at (1.4,-0.8) { Form $6$};
\end{tikzpicture}
\end{minipage}
\end{tabular}}

\scalebox{0.8}{
\begin{tabular}{lcr}
\begin{minipage}{0.38\textwidth}
\begin{tikzpicture}
%\draw[step=1cm,help lines] (-.5,-1) grid (4.5,1);
\filldraw[fill=blue!30,thick] (0,0) rectangle (4.2,0.5);
\filldraw[fill=red!30,thick] (1.25,0) rectangle (3,0.5);
\filldraw[fill=blue!30,thick] (4.2,0) rectangle (5.4,0.5);
\node at (0.65,0.25) {{$\sigma^{i+1}(1)$}};
\node at (2.1,0.25) {{$\sigma^{i+1}(u)$}};
\node at (3.6,0.25) {{$\sigma^{i+1}(1)$}};
\node at (4.8,0.25) {{$\sigma^{i+1}(0)$}};
\filldraw[fill=brown!30,thick] (0.3,0) rectangle (4.8,-0.5);
\filldraw[fill=green!30,thick] (1.1,0) rectangle (3.8,-0.5);
\node at (0.75,-0.25) {{$x$}};
\node at (2.6,-0.25) {{$w^{\prime}$}};
\node at (4.4,-0.25) {{$y$}};
\node at (2.6,-0.8) { Form $7$};
\end{tikzpicture}
\end{minipage} &
\begin{minipage}{0.3\textwidth}
\begin{tikzpicture}
%\draw[step=1cm,help lines] (-.5,-1) grid (4.5,1);
\filldraw[fill=blue!30,thick] (0,0) rectangle (4.2,0.5);
\filldraw[fill=red!30,thick] (1.25,0) rectangle (3,0.5);
\node at (0.65,0.25) {{$\sigma^{i+1}(1)$}};
\node at (2.1,0.25) {{$\sigma^{i+1}(u)$}};
\node at (3.6,0.25) {{$\sigma^{i+1}(1)$}};
\filldraw[fill=brown!30,thick] (0.3,0) rectangle (4,-0.5);
\filldraw[fill=green!30,thick] (1.1,0) rectangle (3.2,-0.5);
\node at (0.75,-0.25) {{$x$}};
\node at (2.1,-0.25) {{$w^{\prime}$}};
\node at (3.6,-0.25) {{$y$}};
\node at (2,-0.8) { Form $8$};
%\draw[<-,thick] (4.3,0.25) -- (5,0.25) node [right] { A fragment of $\mathbf{c}$};
%\draw[<-,thick] (4.3,-0.25) -- (5,-0.25) node [right] { $w=0^{k-1}w^{\prime}x$};
\end{tikzpicture}
\end{minipage} &

\begin{minipage}{0.28\textwidth}
\begin{tikzpicture}
%\draw[step=1cm,help lines] (-.5,-1) grid (4.5,1);
\filldraw[fill=blue!30,thick] (0,0) rectangle (4.2,0.5);
\filldraw[fill=red!30,thick] (1.25,0) rectangle (3,0.5);
\filldraw[fill=blue!30,thick] (4.2,0) rectangle (5.4,0.5);
\node at (0.65,0.25) {{$\sigma^{i+1}(1)$}};
\node at (2.1,0.25) {{$\sigma^{i+1}(u)$}};
\node at (3.6,0.25) {{$\sigma^{i+1}(0)$}};
\node at (4.8,0.25) {{$\sigma^{i+1}(1)$}};
\filldraw[fill=brown!30,thick] (0.3,0) rectangle (4.8,-0.5);
\filldraw[fill=green!30,thick] (1.1,0) rectangle (3.8,-0.5);
\node at (0.75,-0.25) {{$x$}};
\node at (2.6,-0.25) {{$w^{\prime}$}};
\node at (4.4,-0.25) {{$y$}};
\node at (2.6,-0.8) { Form $9$};
\end{tikzpicture}
\end{minipage}
\end{tabular}}
\caption{}\label{fig:3}
\end{figure}

For every fixed pair of $o_{x}\in\tilde{L}_{x}$ and $o_{y}\in\tilde{L}_{y}$, in all the nine forms, we have 
\begin{equation}\label{eq:11n}
n=|w|=3^{i+1}(|\tilde{u}|-1)+\ell(o_{x},o_{y})
\end{equation}
and 
\begin{equation}\label{eq:11w1}
|w|_1=2^{i+1}|\tilde{u}|_1-|\mathrm{pref}_{o_x+z_x-k+2}(\sigma^{i+1}(1))|_{1}-|\mathrm{suff}_{3^{i+1}-o_y-1}(\sigma^{i+1}(1))|_{1},
\end{equation}
where $\ell(o_{x},o_{y}):=(k-1-o_x-z_x+o_y+z_y) <2\cdot 3^{i+1}$ and 
\begin{equation}\label{eq:11tildeu}
\tilde{u} = \begin{cases}
10u1, & \text{ if }w \text{ is of Form }1 \text{ or }5,\\
10u01, & \text{ if }w \text{  is of Form }3,\\
1u01, & \text{ if }w \text{  is of Form }6 \text{ or }9,\\
1u1, & \text{ otherwise.}
\end{cases}
\end{equation}
Further, according to \eqref{eq:2c}, $\tilde{u}$ in \eqref{eq:11tildeu} must satisfy $|\tilde{u}|\equiv 1\mod 2$. This fact and \eqref{eq:11n} yield that $\mathcal{W}_{n,x,y}=\emptyset$ when $n\not\equiv \ell(o_{x},o_{y}) \mod 2\cdot 3^{i+1}$. 

Now we deal with the case $n\equiv \ell(o_{x},o_{y}) \mod 2\cdot 3^{i+1}$. 
Note that by \eqref{eq:2c}, we have $p(2j+1,1,1)=1$ for all $j\geq 1$. This fact and \eqref{eq:11w1} imply that for all  
$n=2\cdot 3^{i+1}j+\ell(o_{x},o_{y})$,  
\[p_k(n,x,y)=\text{Card}(\{|w|_1 \mid w\in\mathcal{W}_{n,x,y}\})=1.\] In conclusion, let $\mathcal{I}_{x,y}=\{2\cdot 3^{i+1}j+\ell(o_{x},o_{y})\mid j\geq 1, o_{x}\in\tilde{L}_{x}, o_{y}\in\tilde{L}_{y}\}$. We have 
\[p_k(n,x,y)=\begin{cases}
1, & \text{if }n\in\mathcal{I}_{x,y},\\
0, & \text{otherwise}.
\end{cases}
\]
Therefore, $\{p_k(n,x,y)\}_{n\geq 1}$ is ultimately periodic with a period $2\cdot 3^{i+1}$.
\end{proof}
\begin{proposition}\label{abel:3}
$\{\mathcal{P}_{\mathbf{c}}^{(k)}(n)\}_{n\geq 1}$ is a $3$-regular sequence for every $k\geq 3.$
\end{proposition}
\begin{proof}
It follows directly from Lemmas \ref{lem:0k-0k}, \ref{lem:0k-x} and \ref{lem:x-y} and (\ref{eqn:decomp}).
\end{proof}
Theorem \ref{thm:A} follows from Propositions \ref{abel:1}, \ref{lem:ab2} and \ref{abel:3}.
\iffalse
\begin{theorem}
$\{\mathcal{P}_{\mathbf{c}}^{(k)}(n)\}_{n\geq 1}$ is $3$-regular for every $k\geq 1$.
\end{theorem}
\begin{proof}
It follows directly from Propositions \ref{abel:1}, \ref{lem:ab2}, \ref{abel:3}.
\end{proof}
\fi
%%%%%%%%%%%%%%%%%%%%%%%%%%%%%%%%%%%%%%%%%%%%%%%%%%%%%%%%%%%%%%%%%%%%%%%%%%%%%%%%%%%%%%%%%%%%%%%%%%%%%%%%%%%%%%%%%%%%%%%%%%%%%%%%%%%%%%%%%%%%%%%%

\end{document}